\documentclass{article}

\usepackage{makeidx}
\usepackage{latexsym}
\usepackage{amsfonts}
\usepackage{amssymb}
\usepackage{amsmath}
\usepackage{amstext}
\usepackage{amsthm}
\usepackage{color}
\usepackage{mathrsfs}
\usepackage{relsize}

\newtheorem{theorem}{Theorem}
\newtheorem{example}[theorem]{Example}
\newtheorem*{claim}{Claim}

\usepackage{amsfonts}
\usepackage{amsthm}
\usepackage{amssymb}
\usepackage{mathrsfs}
\usepackage{mathabx}
\usepackage{tikz}
\usepackage{graphicx}
\usepackage{caption}
\usepackage{subfig}
\usepackage{wasysym}
\usepackage{xcolor}
\usepackage{url}
\usetikzlibrary{calc,intersections}
\newtheorem{lemma}[theorem]{Lemma}
\newtheorem{prop}[theorem]{Proposition}
\newtheorem{corollary}[theorem]{Corollary}
\newtheorem{definition}[theorem]{Definition}

\newtheorem*{main}{Main Theorem}
\newtheorem*{maincor}{Main Corollary}
\usepackage[colorlinks=true, urlcolor=blue, linkcolor=blue, citecolor=blue]{hyperref}
\usepackage{pgfplots}
\pgfplotsset{compat=1.15}
\usepackage{mathrsfs}
\usetikzlibrary{arrows}
\usepackage{algorithm}

\usepackage{quoting}
\quotingsetup{font=normalsize}

\usepackage{verbatim}

\newcommand{\cupdot}{\mathbin{\mathaccent\cdot\cup}}

\begin{document}
	
	\definecolor{qqqqff}{rgb}{0.,0.,1.} 
	\definecolor{bfffqq}{rgb}{0.7490196078431373,1.,0.}
	\definecolor{qqwwtt}{rgb}{0.,0.4,0.2}
	\definecolor{zzccqq}{rgb}{0.6,0.8,0.}
	\definecolor{qqffqq}{rgb}{0.,1.,0.}
	\definecolor{ffqqqq}{rgb}{1.,0.,0.}
	\definecolor{xfqqff}{rgb}{0.4980392156862745,0.,1.}
	\definecolor{ffqqff}{rgb}{1.,0.,1.}
	\definecolor{cqcqcq}{rgb}{0.7529411764705882,0.7529411764705882,0.7529411764705882}
	\definecolor{ffffff}{rgb}{1.,1.,1.}

\title{Critical groups and partitions of finite groups}
 \author{\textbf{Daniela Bubboloni} \\
{\small {Dipartimento di Matematica e Informatica ``Ulisse Dini''} }\\
\vspace{-6mm}\\
{\small {Universit\`{a} degli Studi di Firenze} }\\
\vspace{-6mm}\\
{\small {viale Morgagni, 67/a, 50134, Firenze, Italy}}\\
\vspace{-6mm}\\
{\small {e-mail: daniela.bubboloni@unifi.it}}\\
\vspace{-6mm}\\
{\small tel: +39 055 2759667}\\ \and \textbf{Nicolas Pinzauti}
 \\
\vspace{-6mm}\\
{\small {e-mail: nicolas.pinzauti@edu.unifi.it}}\\
\vspace{-6mm}\\
{\small tel: +39 334 5890261}}

\maketitle

\begin{abstract}
We define a class of finite groups based on the properties of the closed twins of their power graphs and study the structure of those groups. As a byproduct, we obtain results about finite groups admitting a partition by cyclic subgroups.
\end{abstract}

\vspace{4mm}


\vspace{2mm}


\section{Introduction} 

There is a growing literature about graphs associated with groups, as well described by the survey \cite{GraphsOnGroups}. In many cases this results into defining a graph $\Gamma(G)$, having vertex set a group $G$ or a suitable subset of it, and edges or arcs reflecting in some way the nature of the group. Many contributions explore the graph theoretical properties of $\Gamma(G)$, especially for particular classes of groups. See, for instance, \cite{Bubboloni_3}, \cite{shi2}, \cite{Conn}, \cite{Das}, \cite{luch2}. Other contributions study instead how a peculiar structure for $\Gamma(G)$ reflects into the group structure (\cite{Bra},\cite{luch1}, \cite{ForbiddenGraphs}, \cite{Ma2}, \cite{Bri}). Our paper is somewhat  in line with this second approach. However, we think that it is possible to be more ambitious and derive general theoretical results for finite groups exploiting the properties of a certain $\Gamma(G)$ as a main tool. Such idea is developed here with respect to the power graph of a group $G$ (for a recent survey on power graphs see \cite{Kumar}), having in mind the concepts of critical classes and neighbourhood closure introduced in \cite{BP22} in order to emend a mistake in \cite{Cameron_2}.

Critical classes were created as a conceptual tool for the reconstruction of the directed power graph $\vec{\mathcal{P}}(G)$ of a finite group $G$ from its undirected counterpart, the power graph $\mathcal{P}(G)$. They represent the cornerstones of the manageable algorithm for that reconstruction (see \cite{BParxiv}) created in order to answer to \cite[Question 2]{GraphsOnGroups}. Recall that, $\vec{\mathcal{P}}(G)$ has vertex set $G$ and an arc from $x$ to $y$, distinct elements of $G$,
 if $\langle y\rangle \leq \langle x\rangle$. To have the same closed neighbourhood in $\mathcal{P}(G)$ defines an interesting equivalence relation $\mathtt{N}$ in $G$. To generate the same cyclic subgroup defines a further equivalence relation $\diamond$ in $G$ and an $\mathtt{N}$-class is union of $\diamond$-classes. An $\mathtt{N}$-class is called plain if it is a unique $\diamond$-class, otherwise, it is called compound.
A  critical class is  an $\mathtt{N}$-class of special type (see Definition \ref{DefCriticalClass}), difficult to be recognized as plain or compound  through arithmetical considerations on its size and on the size of its so-called neighbourhood closure. This difficulty impacts on the algorithm for the reconstruction of the directed power graph, making it much more articulated and complex.
One can reasonably asks how significant is the presence of such classes in a group. Are they rare or common? Is it possible to understand in which kind of groups they appear and how the structure of $G$ is influenced by them? The present paper approaches those questions starting by the extreme case of groups in which every $\mathtt{N}$-class, apart the class of the neutral element, is critical.
Surprisingly, such groups, which we call {\it critical}, do exist and it is possible to give a complete description of them.
\begin{main}
A group $G$ is critical if and only if $G=\langle x\rangle\rtimes \langle y\rangle$, where $\langle x\rangle=C_{p^a}$, $\langle y\rangle=C_{q^b}$, with $p\neq q$ primes, $a,b\geq 2$, $x^y=x^r,$ with $2\leq r<p^a$ such that $p\nmid r$ and $q^b=|r|_{\mod p}$.
\end{main}
As a consequence of the above theorem, critical groups are a subclass of the so-called EPPO groups, that is, those groups
whose  elements have prime power order.  The class of EPPO groups has a well consolidated history, starting with  Highman \cite{hig} who described the solvable case. Later Suzuki \cite{suz} described the simple case and the classification was finally completed by Brandl \cite{brandl}.
Actually,  the critical groups can be described as the Frobenius non-cyclic metacyclic  {\rm EPPO} groups with $|\pi(G)|\geq 2$ and for which the exponent of a $p$-Sylow subgroup is at least $p^2$,  for all $p\in \pi(G)$ (see Proposition \ref{crit-eppo}).

Since Frobenius groups are a particular class of groups admitting a non-trivial partition, it is also somewhat natural to expect an interaction of our research with partition theory. Indeed, we prove a result about partitions which, interestingly, does not mention criticality or graphs in its statement. It is just a pure group theoretical result.
That indicates that concepts and ideas originally developed to solve the specific question of the reconstruction of the directed power graph  from  the power graph, go very far beyond the original scope. Moreover, this confirms that the ambitious project about the possibility to deduce results of general group theory by the knowledge coming from graphs associate with groups, is not just a dream.
\begin{maincor} Let $G$ be a group 
	admitting a non-trivial partition into cyclic subgroups of orders  a prime power with exponent at least $2$. Then $\pi(G)=\{p,q\}$ with $p,q$ distinct primes and $G$ is a Frobenius group with kernel $F\cong C_{p^a}$ and complement $H\cong C_{q^b}$, for suitable $a,b\geq2$ and $p$ odd.
\end{maincor} 

On the road to prove the Main Corollary, we obtain other results about groups admitting a partition by cyclic subgroups. First, we show that if a group $G$ admits a partition by cyclic subgroups, then every  element of $G$ whose $\mathtt{N}$-class is critical plain generates a maximal cyclic group (Proposition \ref{noPartition}). Then we prove that a group $G$ admitting a non-trivial partition by cyclic subgroups of orders  a prime power with exponent at least $2$ is necessarily critical (Proposition \ref{applicazione_critical}).

The research on partitions is recently quite active (\cite{far}, \cite{gar}, \cite{fog}). Our paper, among other things, aims to give a contribute to this evergreen topic.  

\section{Notation and preliminary results}

We denote by $\mathbb{N}$ the set of positive integers and we set $\mathbb{N}_0:=\mathbb{N}\cup\{0\}$. 
For $k\in \mathbb{Z}$ we set $[k]:=\{x\in \mathbb{N}: x\leq k\}$ and $[k]_0:=\{x\in \mathbb{N}_0: x\leq k\}.$ For $k\in \mathbb{N}$, we denote by $S_k$ the symmetric group on $[k].$
 We say that $n\in \mathbb{N}$ is a prime power [a proper prime power] if there exists a prime number $p$ and $k\in \mathbb{N}_0$ [$k\in  \mathbb{N}$] such that $n=p^k$. So, $1$ is considered a prime power. $\phi$ denotes the Euler's totient function.
 Given $A,B\subseteq X$, we write $X=A\cupdot B$ if $X=A\cup B$ and $A\cap B=\varnothing.$
All the groups considered in this paper are finite.
Given a group $G$, we say that $Y\leq G$ is \emph{cyclic maximal} if there exists no cyclic subgroup $Z$ of $G$ such that $Z>Y.$
An element $x$ of $G$ is called \emph{maximal} if $\langle x \rangle$ is cyclic maximal.
\subsection{Partitions of groups}
Let $G$ be a group with $G\neq 1$. A partition of $G$ is a set $\mathcal{P}$ of non-trivial subgroups of $G$, such that every element $x\in G\setminus\{1\}$ belongs to a unique subgroup in $\mathcal{P}.$ The subgroups in $\mathcal{P}$ are called the components of the partition. 
 If $|\mathcal{P}|=1,$ the partition is called the trivial partition and its only component is $G.$
For instance, a non-trivial cyclic group admits only the trivial partition.  
Given $U_i\leq G$, for $i\in [k]$ and $k\in \mathbb{N},$
 it is immediately observed that  $\mathcal{P}=\{U_i: i\in [k]\}$ is a non-trivial partition of $G$ if and only if $k\geq 2$, $1<U_i<G$ for all $i\in [k]$, $G=\bigcup_{i=1}^kU_i$ and $U_i\cap U_j=\{1\}$ for all $i\neq j\in [k]$.

In this paper, we are particularly interested in partitions with cyclic components and partitions of $p$-groups. It is easily checked that if $\mathcal{P}$ is a partition  with cyclic components, then those components  are  cyclic maximal.
In the theory of partitions of $p$-groups a main role is played by the Hughes-Thompson subgroup.
Recall that, for a $p$-group $G$, such subgroup is defined by $$H_p(G)\coloneq\langle x\in G: o(x)\neq  p\rangle.$$
\begin{theorem}[Kegel] \label{Kegel} A $p$-group $G$ admits a non-trivial partition if and only if $H_p(G)\neq G.$
\end{theorem}
A sequel of results by Kegel, Baer and Suzuki  has led to a complete classification of the finite groups admitting a non-trivial partition.
A group $G$ admits a non-trivial partition if and only if $G$ belongs to the following list (see, for instance, \cite[Section $7$]{zappa}):
\begin{itemize}
\item[(1)] a $p$-group with $H_p(G)\neq G$, 
\item[(2)] a group of Hughes-Thompson type,
\item[(3)] a Frobenius group,
\item[(4)] $\mathrm{PSL}_2(q)$,
\item[(5)] $\mathrm{PGL}_2(q)$ with $q$ odd,
\item[(6)] a Suzuki group $Sz(2^{2m+1})$, with $m\geq 1.$
\end{itemize}
 In particular,  $\mathrm{D}_{2n}$ with $n$ odd admits a partition, because it is a Frobenius group.
 \subsection{Power graphs}
  
Let $G$ be a group. The \emph{power graph} of $G$, denoted by $\mathcal{P}(G)$, has vertex set $G$ and edge set $E$, where for $x\neq y\in G$,  $\left\lbrace x,y \right\rbrace \in E$ if $\langle x\rangle\leq \langle y\rangle$ or $\langle y\rangle\leq \langle x\rangle.$
A graph $\Gamma$ is said to be a power graph if there exists a group $G$ such that $\Gamma=\mathcal{P}(G)$. 

 For $x\in G$, the \emph{closed neighbourhood} of $x$ in $\Gamma=\mathcal{P}(G)$ is given by  $N[x]=\left\lbrace y \in G | \left\lbrace x, y \right\rbrace \in E\right\rbrace\cup\{x\} $. 
Note that, for every $x,y\in G$, we have $y\in N[x]$ if and only if $x\in N[y]$. If $N[x]=G$, then $x$ is called a \emph{star vertex} of $\Gamma$. The set of star vertices of $\Gamma$ is denoted by $\mathcal{S}$. 
There are two equivalence relations on $G$ on which we are concerned in this paper.
For $x,y\in G$, we write $x \mathtt{N}y$ if $N[x]=N[y]$.  The relation $\mathtt{N}$ is an equivalence relation called the \emph{closed twin relation}. Note that $x \mathtt{N}y$, with $x\neq y$ implies $\{x,y\}\in E$.
We denote the $\mathtt{N}$-class of $x\in G$ by $[x]_{\mathtt{N}}$.
In \cite[Section 3]{BP22}, given $X\subseteq G$, it is considered the common closed neighbourhood $N[X]$, given by the intersection of the closed neighbourhoods of the vertices in $X$, when $X\neq\varnothing$ and by $G$ when $X=\varnothing$. Moreover, it is introduced the neighbourhood closure  $\hat{X}\coloneq N[N[X]]$. Remarkably,  the neighbourhood closure defines a Moore closure on $G$ whose main properties  are collected in \cite[Proposition 2]{BP22}.	
 For $x,y\in G$, we write  $x\diamond y$ if $\langle x\rangle=\langle y\rangle.$
The relation $\diamond$ is an equivalence relation, refining $\mathtt{N}$. Every $\mathtt{N}$-class is therefore a union of $\diamond$-classes.

We  now recall some peculiar definitions and results from the literature about  power graphs on which our paper relies.

\begin{definition}[\cite{Cameron_2}, \cite{BP22}]
		\label{NplainType}\rm Let $G$ be a group and $C$ be an $\mathtt{N}$-class. We say that $C$ is a \emph{ class of plain type} if 
		$C$ is a single  $\diamond$-class; $C$ is a \emph{class of compound type} if $C$ is the union of at least two $\diamond$-classes.
	\end{definition}
As an interesting example, by an immediate application of \cite[Proposition 4]{Cameron_2}, we can clear the nature of the class $\mathcal{S}=[1]_{\mathtt{N}}$ in the power graphs. 
\begin{example}\label{S-comp} {\rm Let $G$ be a group. Then the following facts hold:
		\begin{enumerate}
			\item[$(i)$] If $|\mathcal{S}|=1$, then $\mathcal{S}$ is the $\mathtt{N}$-class of plain type $\{1\}$; 
			\item[$(ii)$] If $|\mathcal{S}|\geq 2$, then $\mathcal{S}$ is an $\mathtt{N}$-class of compound type. More precisely, one of the following holds:
			\begin{enumerate}
				\item[\rm (a)] $G\cong C_{p^n}$ for some prime number $p$, $n\in \mathbb{N}$, and  $\mathcal{S}=G$ is the union of $(n+1)$ $\diamond$-classes, each containing the elements of order $p^i$ in $G$ for $i\in[n]_0$; 
				\item[\rm (b)] $G$ is cyclic, not a $p$-group, and $\mathcal{S}$ is the union of two $\diamond$-classes given by $\{1\}$ and the set of generators of $G$;
				\item[\rm (c)] $G$ is a generalized quaternion $2$-group and $\mathcal{S}$  is the union of two $\diamond$-classes given by $\{1\}$ and $\{\tau\}$, where $\tau$ is 
				the unique involution of $G$.
			\end{enumerate} 
		\end{enumerate} 	
		In particular, $\mathcal{S}$ is of plain type if and only if $\mathcal{S}=\{1\}.$}
	\end{example}
The $\mathtt{N}$-classes in $G$ different from $\mathcal{S}=\{1\}$ will be called the non-trivial $\mathtt{N}$-classes.
	
We now summarize some results and definitions from \cite{Cameron_2} and \cite{BP22} about compound classes.	
\begin{prop}{\rm [\cite[Proposition 5]{Cameron_2}, \cite[Proposition 15]{BP22}]}
		\label{propC_y}
		Let $G$ be a group and $C$ be an $\mathtt{N}$-class of $G$, with $C\neq \mathcal{S}$. The following facts are equivalent:
		
		\begin{enumerate}
		\item[$(i)$] $C$ is a compound class;
			
		\item[$(ii)$] If $y$ is an element of maximum order in $C$, then $o(y)=p^r$ for some prime number $p$ and some integer $r \geq 2$. Moreover, there exists $s\in [r-2]_{0} $ such that $$C= \left\lbrace z \in \langle y\rangle \, |\, p^{s+1}\leq o(z)\leq p^r \right\rbrace. $$ 
		\end{enumerate} 	
		If $C\neq \mathcal{S}$ is a compound class, the ordered list $(p,r,s)$ is uniquely determined by $C$ and called the parameters of $C$; $y$ as in $(ii)$ is called a root of $C;$ the size of $C$ is $p^r-p^s$;  if $x\in G$ is such that $[x]_\mathtt{N}=C$, then $o(x)$ is a proper prime power.
			\end{prop}
The neighbourhood closure comes into play for the following crucial definition.
\begin{definition}
	\label{DefCriticalClass}
			\rm	Let $\Gamma$ be a power graph. A \emph{critical class} is an $\mathtt{N}$-class $C$ such that $\hat{C}=C \cupdot \{1\}$ and there exist a prime number $p$ and an integer $r\geq 2$, with $|\hat{C}|=p^r$. 
					\end{definition}
Note that, in particular, within the definition of critical class $C$, we are requiring that its neighbourhood closure is the minimum possible. Indeed, by \cite[Proposition 2]{BP22}, we know that $\hat{C}\supseteq C\cup \mathcal{S}\supseteq C\cup \{1\}$.			
		\begin{example}\label{controesempio}{\rm \cite[Examples 12 and 15]{BP22}}
		{\rm In $G=\mathrm{D}_{30}$ the ${\mathtt{N}}$-class of an element of order $15$ is plain and critical. In $S_4$ the $\mathtt{N}$-class of a $4$-cycle is compound and critical.
	The class $\mathcal{S}$ is never critical because it contains $1.$
}
\end{example}

	Classes of compound type that are critical are easily characterized.
	
\begin{prop}\label{par-crit-comp}{\rm\cite[pag. 11]{BP22}} Let $G$ be a group and $C\ne \mathcal{S}$ be an $\mathtt{N}$-class of compound type. $C$ is critical if and only if $C$ has parameters $(p,r,0)$.
\end{prop}
We finally state a sufficient condition for a plain class to be critical.
\begin{prop} \label{CarachetisationN-classes_2}{\rm \cite[Proposition 14]{BP22}}	Let $G$ be a group and $C$ an $\mathtt{N}$-class of plain type. Assume that there exist a prime $p$ and integers $r\geq 2$ and $s\in [r-2]_0$ such that  $|C|=p^r-p^s$ and $|\hat{C}|=p^r$. Then $C$ is critical, $s=0$, and
$C=[y]_{\diamond}$ for some $y\in G$, with $o(y)>1$ not a prime power and such that $\phi(o(y))=p^r-1$.
\end{prop}
		

		\section{Critical elements}
		
		We start our research by defining the concept of critical, plain and compound element in a finite group.
		\begin{definition}
	\label{DefCriticalGrouElements}
			\rm	Let $G$ be a group. We say that $ x\in G$ is 
			\begin{itemize}
			\item[1.] \emph{ critical} if $[x]_{\mathtt{N}}$ is a critical class in $\mathcal{P}(G)$;
			\item[2.] \emph{plain} [\emph{compound}] if $[x]_{\mathtt{N}}$ is a plain [compound] class in $\mathcal{P}(G)$.
\end{itemize}

					\end{definition}

If $x\in G$ is a plain [compound] element and also a critical element, then we simply call it \emph{plain} [\emph{compound}] \emph{critical}.
Obviously,  $x$ is critical if and only if $x$ is plain critical or compound critical. Since the class $[1]_{\mathtt{N}}=\mathcal{S}$ is never critical, $1$ is never critical.

\begin{prop}\label{crit_S_1}
	Let $G$ be a group. If there exists $x\in G$ critical, then $\mathcal{S}=\{1\}$ and $G\neq 1.$
\end{prop}
\begin{proof}
	Let $x\in G$ be critical and consider $X=[x]_{\mathtt{N}}$. Then we have $\hat{X}=X \cupdot \{1\}$. In particular, $1\notin X$. Assume, by absurd, that $\mathcal{S}\ne \{1\}$. Then there exists $s\in \mathcal{S}\setminus \{1\}$. Since, by \cite[Proposition 2]{BP22}, we have $\mathcal{S}\subseteq \hat{X}= X \cupdot \{1\}$, then $s \in X$. As a consequence, we have $X=\mathcal{S}=[1]_{\mathtt{N}}$ and hence $1\in X$, a contradiction. Finally note that, since $x$ is critical we know that  $x\neq 1$. Thus $G\neq 1.$
\end{proof}

The previous proposition is a first basic  result that links  the presence of critical elements with group properties. One of the scopes of this section is to deepen this link. To  that purpose, we first explore  the crucial links between criticality of elements and their orders.

\begin{prop}\label{ord-pl-cr} Let $G$ be a group and $x\in G$ be critical. Then $x$ is compound if and only if $o(x)$ is a proper prime power.
\end{prop}

\begin{proof} Assume that $x$ is compound. Then $[x]_{\mathtt{N}}\neq \mathcal{S}$ is compound critical. Hence, by Proposition \ref{propC_y}, $o(x)$ is a proper prime power. Conversely assume that $o(x)$ is a proper prime power. Since $[x]_{\mathtt{N}}$ is critical, we have $|[x]_{\mathtt{N}}|=p^r-1,$ for some prime $p$ and $r\geq 2$ integer. Then Proposition \ref{CarachetisationN-classes_2} applies, giving $o(x)$ not a prime power, a contradiction.	 
\end{proof}


\begin{lemma}\label{oss-crit-2UP}
	Let G be a group and $x\in G$, with $o(x)$ not a prime power and $G\ne \langle x \rangle$. Then $x$ is plain and the following facts are equivalent:
	
	\begin{itemize}
		\item[$(i)$] $\widehat{[x]_{\mathtt{N}}}= [x]_{\mathtt{N}}\cupdot \{1\}$;
		\item[$(ii)$] for every $y\in G$ such that $\langle x\rangle < \langle y \rangle$, there exists $z \in G\setminus N[y]$ such that $\langle x\rangle < \langle z \rangle$. 
	\end{itemize}
\end{lemma}
\begin{proof} Let  $X:= [x]_{\mathtt{N}}$. To start with, we show that $X\neq \mathcal{S}$. Assume the contrary. Then, by Example \ref{S-comp}, one of the following holds: $x=1$, $o(x)=2$ or $\langle x \rangle=G$. However, $x=1$, $o(x)=2$ are incompatible with $o(x)$ not a prime power  and $\langle x \rangle=G$ is excluded by hypothesis. As a consequence, we can apply Proposition \ref{propC_y} and deduce that $x$ is plain. We now show that $(i)$ and $(ii)$ are equivalent.
\smallskip

$(i)\Rightarrow (ii)$	Assume that $\hat{X}= X\cupdot \{1\}$.  Let $y\in G$ be such that $\langle x\rangle < \langle y \rangle$. Then
$x\in \langle y \rangle$ and $o(x)<o(y)$.
	Since $x$ is plain, we have $\hat{X}=[x]_{\diamond}\cupdot \{1\}$ and hence $y\notin \hat{X}$. By \cite[Proposition 2]{BP22}, we have that  $\hat{X}= \bigcap_{z\in N[x]}N[z]$.
	As a consequence, there exists $z\in N[x]$ such that $y\notin N[z]$.
	By the fact that $z\in N[x]$, we have that $z\in \langle x \rangle$ or $x \in \langle z \rangle$.
	Assume that $z \in \langle x \rangle$. Then, since $x \in \langle y \rangle$, we deduce $z \in \langle y \rangle$ and hence $y\in N[z]$, a contradiction.
	Hence $x\in \langle z \rangle$ and $o(x)< o(z)$. 
	Finally note that $z \notin N[y]$ follows by $y \notin N[z]$.
\smallskip

$(ii)\Rightarrow (i)$
	Assume that, for every $y\in G$ such that $\langle x\rangle < \langle y \rangle$ there exists $z \in G\setminus N[y]$ such that $\langle x\rangle < \langle z \rangle$.
	 We know that $X\ne \mathcal{S}$  and thus  $1\notin X$. By \cite[Proposition 2]{BP22} we also know that $\hat{X}\supseteq X \cup \{1\}$.
	So, we only need to check that $\hat{X}\subseteq X \cup \{1\}$.
	Assume, by absurd, that there exists $y\in \hat{X}\setminus(X\cup \{1\})$.
	Then  $y \in N[x]$ and  $y$ is not a generator of $\langle x \rangle$ nor the identity. Assume first that $\langle y \rangle <  \langle x \rangle$. Then $y\in \langle x \rangle$ is not a star vertex for $\mathcal{P}(\langle x \rangle)$ and hence there exists $z\in N[x]\supseteq \langle x \rangle$ not joined with $y$. Since, by \cite[Proposition 2]{BP22}, we have $\hat{X}= \bigcap_{z\in N[x]}N[z]$, we deduce that $y\notin \hat{X}$, a contradiction. Assume next that  $\langle x \rangle < \langle y \rangle$.
	Then, by hypothesis, there exists $z \in G\setminus N[y]$ such that $\langle x \rangle < \langle z \rangle$.
	It follows that $z \in N[x]$ and $y \notin N[z]$. So $y \notin \hat{X}$, a contradiction.
\end{proof}

\begin{prop}\label{crit-2UP_v2}
	Let $G$ be a group and $x\in G$, with $o(x)$ not a prime power and $G\ne \langle x \rangle$. Assume further that  $\phi(o(x))= p^r-1$, for some prime $p$ and some integer $r \geq 2$. Then
	the following facts are equivalent:
	\begin{itemize}
		\item[$(i)$] $x$ is plain critical;
		\item[$(ii)$] for every $y\in G$ such that $\langle x\rangle < \langle y \rangle$, there exists $z \in G\setminus N[y]$ such that $\langle x\rangle < \langle z \rangle$. 
	\end{itemize}
\end{prop}
\begin{proof}
	$(i)\Rightarrow (ii)$ Assume that $x$ is plain critical. Then $\widehat{[x]_{\mathtt{N}}}= [x]_{\mathtt{N}}\cupdot \{1\}$ holds. Therefore, by Lemma \ref{oss-crit-2UP}, we have that also $(ii)$ holds.
	\smallskip
	
	$(ii)\Rightarrow (i)$ Assume that, for every $y\in G$ such that $\langle x\rangle < \langle y \rangle$ there exists $z \in G\setminus N[y]$ such that $\langle x\rangle < \langle z \rangle$. By Lemma \ref{oss-crit-2UP} we have that $x$ is plain and that $\widehat{[x]_{\mathtt{N}}}= [x]_{\mathtt{N}}\cupdot \{1\}$ holds. It follows that $\widehat{[x]_{\mathtt{N}}}= [x]_{\diamond}\cupdot \{1\}$. Then, using   $\phi(o(x))=p^r-1$, we obtain $|\widehat{[x]_{\mathtt{N}}}|= |[x]_{\diamond}|+1= p^r$. Therefore $[x]_{\mathtt{N}}$ is plain critical and hence $x$ is plain critical.
\end{proof}

From the above proposition we deduce a number of results that are useful for at least two reasons. They allow to recognize plain critical elements and to obtain an unexpected application to the theory of partitions (see Proposition \ref{noPartition}).

\begin{corollary}\label{crit-NoUP}
	Let $G\neq 1$ be a group with $\mathcal{S}=\{1\}$, and let $x\in G$ be maximal with $o(x)$ not a prime power. Then $\widehat{[x]_{\mathtt{N}}}=[x]_{\mathtt{N}} \cupdot \{1\}$ and $x$ is plain.
	If further $\phi(o(x))=p^r-1$ for some  prime $p$ and some integer $r\geq2$ holds, then $x$ is critical.
\end{corollary}
\begin{proof}  By $\mathcal{S}=\{1\}$ and $G\neq 1$, we deduce that $G$ is not cyclic. Hence, by Lemma \ref{oss-crit-2UP}, $x$ is plain. Since $x$ is maximal,  the condition $(ii)$ of Lemma \ref{oss-crit-2UP} is trivially satisfied. Thus we have $\widehat{[x]_{\mathtt{N}}}=[x]_{\mathtt{N}} \cupdot \{1\}$.\\
 Assume further that  $\phi(o(x))=p^r-1$  holds,  for some  prime $p$ and some integer $r\geq2$. Since the condition $(ii)$ in Proposition \ref{crit-2UP_v2} is satisfied, we deduce that $x$ is critical.
\end{proof}
The next example shows that it is possible to have $x\in G$
 plain critical with $\phi(o(x))=p^r-1$,  for some  prime $p$ and some integer $r\geq2$, with $x$ not maximal.
 \begin{example}
 	The permutation $\sigma:= (123)(45678)\in S_{11}$ is plain critical  and non-maximal.
 \end{example}
 \begin{proof}
 	Firstly note that $\langle \sigma \rangle < \langle (1 2 3)(4 5 6 7 8)(9 \ 10) \rangle$ and hence $\sigma$ is not maximal.
	Moreover, we have $o(\sigma)=15$ and $\phi(15)=3^2-1$. Let $\alpha\in S_{11}$ be such that $\langle \sigma\rangle < \langle \alpha \rangle$. Then $\alpha=\sigma \tau$ for some $\tau\in \{(9 \ 10), (9\ 11), (10\ 11)\}$. Note that $\alpha$ admits a unique fix point. Consider then $\gamma\coloneq \sigma \tau'$ with $\tau'\neq \tau$. Then $\alpha$ and $\gamma$ have same order and different fix points. It follows that $\gamma\in S_{11}\setminus N[\alpha]$ and thus, by Proposition \ref{crit-2UP_v2}, $\sigma$ is plain critical and not maximal.
 \end{proof}

It seems important to better investigate the consequences of the existence of elements which are plain critical but not maximal.
In order to explore that issue, we introduce a minor player of our research, the \emph{enhanced power graph} of $G$, denoted by $\mathcal{E}(G)$. Its vertex set is $ G$ and for $x\neq y\in G$, $\left\lbrace x,y \right\rbrace $ is an edge if  $\langle x,y\rangle$ is cyclic.
\begin{corollary}\label{corollary_crit-2UP}
	Let $G$ be a group and $x\in G$ be plain critical and not maximal. Then there exist $y\in G$ and $z\in G \setminus N_{\mathcal{E}(G)}[y]$ such that  $\langle x \rangle < \langle y \rangle$ and $\langle x \rangle < \langle z \rangle$.
\end{corollary}
\begin{proof}
	Since $x$ is plain critical, by Proposition \ref{ord-pl-cr}, we have that $o(x)$ is not a prime power. 
Moreover, from $[x]_{\mathtt{N}}=[x]_{\diamond}$, $\widehat{[x]_{\mathtt{N}}}=[x]_{\diamond}\cupdot \{1\}$ and $|\widehat{[x]_{\mathtt{N}}}|=p^r$, for some prime $p$ and some integer $r\geq2$, we deduce that $\phi(o(x))=p^r-1$. Finally note that $\langle x \rangle \ne G$ as otherwise we would have  $[x]_{\mathtt{N}}= \mathcal{S}$ against the fact that $\mathcal{S}$ is never critical. As a consequence, the hypothesis of Proposition \ref{crit-2UP_v2} are satisfied.
Since $x$ is not maximal, there exists  $y_1\in G$ such that $\langle y_1 \rangle >\langle x \rangle$. By Proposition \ref{crit-2UP_v2}, there exists $z_1 \in G \setminus N[y_1]$ such that $\langle z_1 \rangle >\langle x \rangle$.  If $\langle y_1, z_1 \rangle$ is not cyclic we define $y\coloneq y_1$ and $z\coloneq z_1$ and we have finished. 

Assume next that $\langle y_1, z_1 \rangle$ is cyclic. Let $y_2\in G$ be such that $\langle y_1, z_1 \rangle=\langle y_2\rangle.$ Note that $\langle y_2\rangle>\langle y_1\rangle$ because, since $z_1 \notin N[y_1]$, then $z_1$ is not a power of $y_1.$ 
Hence we have produced the chain $\langle y_2\rangle>\langle y_1\rangle>\langle x \rangle.$ We now apply again Proposition \ref{crit-2UP_v2} to $y_2$, finding  $z_2 \in G \setminus N[y_2]$ such that $\langle z_2 \rangle >\langle x \rangle$. If $\langle y_2, z_2 \rangle$ is not cyclic we define $y\coloneq y_2$ and $z\coloneq z_2$, obtaining $z_2\in G \setminus N_{\mathcal{E}(G)}[y_2].$

If instead $\langle y_2, z_2 \rangle$ is cyclic generated by a certain $y_3\in G$, we produce a longer chain $\langle y_3\rangle>\langle y_2\rangle>\langle y_1\rangle>\langle x \rangle.$ 
Since $G$ is finite, such a chain cannot increase arbitrary in length. Hence after $n$ steps, for a suitable $n\in\mathbb{N}$, we produce $y_n\in G$ and $z_n \in G\setminus N_{\mathcal{E}(G)}[y_n]$ with 
$\langle z_n \rangle >\langle x \rangle$ as required.
\end{proof}
\begin{prop}\label{noPartition}
	Let $G$ be a group. If $G$ admits a partition by cyclic subgroups, then every  plain critical element is maximal.
\end{prop}
\begin{proof}
	Let $\mathcal{P}$ be  a partition of $G$ whose components are cyclic subgroups of $G$. Assume, by contradiction, that there exist $x\in G$  plain  critical and not maximal. Since $x$ is critical, we have that $x\neq 1.$ By Corollary \ref{corollary_crit-2UP}, there exist $y,z\in G$ such that  $\langle x \rangle < \langle y\rangle$, $\langle x \rangle < \langle z \rangle$ and $\langle y,z\rangle$ is not cyclic. 

Let $C\in \mathcal P$ be such that $y\in C$ and $C' \in \mathcal{P}$ be such that $z \in C'$.
Then both $C$ and $C'$ contains $x\neq 1$ and thus $C=C'$. As a consequence, $C$ is  a cyclic subgroup of $G$ containing both $y$ and $z$, which implies $\langle y,z\rangle$ cyclic, a contradiction.
\end{proof}

We emphasize that if a group admits a partition by cyclic subgroups, plain critical elements may or not exist.

Consider the dihedral group $\mathrm{D}_{2n}$, $n\geq 2$. As well-known,  $\mathrm{D}_{2n}$ admits a partition by cyclic subgroups. However, $\mathrm{D}_{2n}$ contains plain critical elements if and only if $n$ is not a prime power and $\phi(n)=p^r-1$ holds, for some prime $p$ and some integer $r\geq2$. Indeed, assume that $\mathrm{D}_{2n}$ contains a plain critical element $x$.
By Proposition \ref{ord-pl-cr}, the involutions are never plain critical. Then $n\geq 3$ and $x$ belongs to the normal subgroup of size $n$ in $\mathrm{D}_{2n}$.
Since $x$ is maximal, we then have $o(x)=n$ and thus, by Proposition \ref{ord-pl-cr}, $n$ is not a prime power. Moreover, we have $|\widehat{[x]_{\mathtt{N}}}|=p^r$, for some prime $p$ and some integer $r\geq2$. Since we also have  $\widehat{[x]_{\mathtt{N}}}=[x]_{\diamond}\cupdot \{1\}$, we get 
$\phi(n)=p^r-1$.
Conversely, assume that $n$ satisfies those arithmetic properties. The elements of order $n$ in $\mathrm{D}_{2n}$ are maximal 
and hence, by Proposition \ref{crit-2UP_v2}, they are plain critical. 

\section{Critical groups}

We now define three typologies of groups.
		\begin{definition}
	\label{DefCriticalGroups}
			\rm	Let $G\neq 1$ be a group. We say that $ G$ is 
			\begin{itemize}
			\item[1.] \emph{critical} if every $x\in G\setminus\{1\}$ is critical; 
			\item[2.] \emph{plain} [\emph{compound}] if every $x\in G\setminus\{1\}$ is plain [compound];
\end{itemize}
\end{definition}

A group that is both plain [compound] and critical will be called plain  critical [compound critical]. Note that, by Proposition \ref{crit_S_1}, if $G$ is critical then $\mathcal{S}=\{1\}$.
					
In principle, it is not clear if critical groups, plain groups and compound groups do exist. In fact, the existence of plain groups is easily established.
\begin{prop}\label{abelian-plain}
	Every non-cyclic abelian group is plain.
\end{prop} 
\begin{proof}
It  follows immediately by the fact that in a non-cyclic abelian group every $\mathtt{N}$-class is a $\diamond$-class as was shown in \cite[Proposition 2]{Cameron_1}. 
\end{proof}

\subsection{Existence of compound critical groups and partitions}
In order to prove that critical groups and compound groups  exist we are going to show that compound critical groups do exist. 

\begin{prop}\label{applicazione_critical} Let $G$ be a group 
	admitting a non-trivial partition into cyclic subgroups of orders  a prime power with exponent at least $2$. Then $G$ is compound critical.
\end{prop}
\begin{proof} 
	Let $\mathcal{P}$ be a partition of $G$ with cyclic components. Then, for every $U\in \mathcal{P}$, we have $1<U<G.$ Consider the set 
	$$\mathcal{C}:=\{C
	: C \hbox{ is a } \mathtt{N}\hbox{-class of } G\}\setminus \{\mathcal{S}\}$$
	of the non-trivial $\mathtt{N}$-classes of $G$,
	and the set
	$$\hat{\mathcal{C}}:=\{\hat{C}: C \in \mathcal{C}\}$$
	of their neighbourhood closures.
	We claim that 
	\begin{equation}\label{prima1}
		\hat{\mathcal{C}}=\mathcal{P}
	\end{equation}
	and that 
	\begin{equation}\label{seconda2}
		\mathcal{C}=\{U\setminus\{1\}: U\in \mathcal{P}\}.
	\end{equation}
	
	We first show that if $x\in G\setminus\{1\}$, then $N[x]= U$, where $U$ is the unique component in $\mathcal{P}$ such that $x\in U$.
	Let $x\in G\setminus\{1\}$ and $U$ be the unique component in $\mathcal{P}$ such that $x\in U$.
	Since $U=\langle u\rangle$ is a cyclic group of prime power order, we immediately have that $N[x]\supseteq U$. We show that $N[x]\subseteq U$.
	Let $y\in N[x]$. If $y$ is a power of $x$, then $y\in U.$ Assume next that $x=y^m$, for some positive integer $m$. Suppose, by contradiction, that $y\notin U$. Then $y\in V$ for a  component $V\in\mathcal{P}\setminus\{U\}$. It follows that $1\neq x\in U\cap V$, against the definition of partition.
	
	Moreover, we deduce that $N[x]=N[y]$ for $x, y\in G\setminus\{1\}$ if and only if $x$ and $y$ belong to the same component of the partition $\mathcal{P}.$ 
	Thus, for every $x\in G\setminus\{1\}$, we have $[x]_{ \mathtt{N}}=U\setminus\{1\}$, where $U$ is the unique component in $\mathcal{P}$ such that $x\in U$. On the other hand, given $U\in\mathcal{P}$ and taken $x\in U\setminus\{1\}$, we obviously have $[x]_{ \mathtt{N}}=U\setminus\{1\}$. Hence we have proved \eqref{seconda2}.
	
	In order to show \eqref{prima1}, it remains to show that, for every $U\in \mathcal{P}$ and $x\in U\setminus\{1\},$ we have $\widehat{[x]_{ \mathtt{N}}}=U.$   
	Indeed, by \cite[Proposition 2]{BP22}, we have 
	$$(U\setminus\{1\})\cup \{1\}\subseteq \widehat{[x]_{ \mathtt{N}}}\subseteq N[x]=U,$$
	so that $\widehat{[x]_{ \mathtt{N}}}=U.$ 
	
	Note that every $x\in G\setminus\{1\}$ is critical, because we know that $\widehat{[x]_{ \mathtt{N}}}=[x]_{ \mathtt{N}}\cupdot\{1\}=U$ and $|U|=p^r$, for some prime $p$ and some integer $r\geq 2$.
	Moreover, since $U$ is a cyclic subgroup, in $[x]_{\mathtt{N}}= U \setminus \{1\}$ there are elements of order $p$ and of order $p^r$ and hence $x$ is compound. Therefore $G$ is compound critical.
\end{proof}

We introduce now some notation and recall some basic facts. Let $m\in \mathbb{N}$ and consider the ring $\mathbb{Z}_m=\mathbb{Z}/(m)$ of the integers modulo $m$. Let $U_m:=U(\mathbb{Z}_m)$ be the group of its units. Then $U_m$ is abelian of order $\phi(m)$. For $x\in \mathbb{Z}$, we have that $x+(m)\in U_m$ if and only if $x$ is coprime with $m$.  We denote by $|x|_{\mod m}$ the order of $x+(m)\in U_m$ in $U_m.$
If $p$ is a prime and $n\in \mathbb{N}$, we have  $U_{p^n}\cong C_{p^{n-1}(p-1)}$. The map $\varphi: U_{p^n}\rightarrow U_p$, given by $\varphi(x+(p^n))=x+(p)$ for $x\in \mathbb{Z}$ coprime with $p$,  is well defined and is a surjective group homomorphism. It follows that  $o(x+(p))$ divides $o(x+(p^n))$, that is 
$|x|_{\mod p}$ divides $ |x|_{\mod p^n}$. For instance $|2|_{\mod 3}=2$ divides $|2|_{\mod 9}=6.$

\begin{prop}\label{Frobenius-espliciti}  Let $C_{p^a}=\langle x\rangle$ and $C_{q^b}=\langle y\rangle$ with $p\neq q$ primes and $a,b\geq 1$ integers. Then the following facts hold:
	\begin{itemize}
		\item[$(i)$] the semidirect product
		$C_{p^a}\rtimes C_{q^b}$ defined by the condition $x^y=x^r$, with $2\leq r<p^a$ such that $p\nmid r$, 
		is a Frobenius group with kernel $C_{p^a}$ if and only if $q^b=|r|_{\mod p};$
		\item[$(ii)$] the Frobenius groups in $(i)$ exist if and only if $q^b\mid p-1$.
	\end{itemize}
\end{prop}
\begin{proof} $(i)$ The condition that guarantees that $G=C_{p^a}\rtimes C_{q^b}$ is a Frobenius group with kernel $C_{p^a}$ is the following
	\begin{equation}\label{Frob}
		(x^k)^{y^s}=x^k\quad \hbox{for some}\quad 1\leq k\leq p^a,\ 1\leq s\leq q^b\ \Longrightarrow k=p^a \quad \hbox{or} \quad s=q^b.
	\end{equation}
	Assume that $G$ is a Frobenius group. Then, by \eqref{Frob}, for $s<q^b$ we have $(x^{p^{a-1}})^{r^s}\neq x^{p^{a-1}},$ that is, $x^{p^{a-1}(r^s-1)}\neq 1$ which implies $p\nmid 
	r^s-1$ so that $r^s\not\equiv 1 (\mbox{mod }p)$. On the other hand, we have 
	$$x^{r^{q^b}}=x^{y^{q^b}}=x$$ so that $p\mid r^{q^b}-1$, hence $r^{q^b} \equiv 1 \mbox{ mod }p$ and thus $q^b=|r|_{\mod p}.$ 
	Conversely, assume that $q^b=|r|_{\mod p}.$ We show that \eqref{Frob} holds. Equality $(x^k)^{y^s}=x^k$ implies $x^{kr^s-k}=1$ and hence $p^a\mid k(r^s-1).$ Assume $s<q^b$. Then $p\nmid r^s-1$ and we obtain $p^a\mid k.$ Then $p^a\leq k\leq p^a,$ so that $k=p^a$.
	\medskip
	
	$(ii)$ In order to have Frobenius groups in $(i)$ we just need to guarantee the existence of $r\in \mathbb{N}$ with $p\nmid r$ such that the order of $r+(p)$ in $U_p$ is $q^b$. Since $U_p\cong C_{p-1}$ is cyclic this happens if and only if $q^b\mid p-1$.
\end{proof}

We can now exhibit an infinite family of explicit examples of compound critical groups.

\begin{example}\label{Esempio_gruppo_critico}
	{\rm Let $a$ be an integer with $a\geq 2$ and $C_{5^a}=\langle x\rangle$.  Let  $C_4=\langle y\rangle$ denote the unique subgroup of $Aut(C_{5^a})$ of order $4$, and let $G:= C_{5^a}\rtimes C_{4}$ be defined by the condition $x^y=x^7$. Note that $4=|7|_{\mod 5}$. Then, by Proposition \ref{Frobenius-espliciti}, $G$ is a Frobenius group. As a consequence, $G$ admits a non-trivial partition into cyclic subgroups having orders $5^a$ and $2^2$.
Then, by Proposition \ref{applicazione_critical}, $G$ is compound critical.  }
\end{example}

\subsection{The nature of critical groups and their characterization}
We want to study in more detail the critical groups at the scope of their characterization.
 An easy result is the following.

\begin{prop}
	\label{NoPlainCritical}
There exists no plain critical group.
\end{prop}
\begin{proof} Assume, by contradiction, that there exists a plain critical group $G$. Since $G\neq 1$, there exists $p\mid |G|$ and
$y\in G$ an element of order $p$. Since $G$ is  plain critical, we have that $y$ is plain. But the fact that $y$ has order a proper prime power goes against Proposition \ref{ord-pl-cr}.
\end{proof}
As a consequence of Proposition \ref{NoPlainCritical}, a critical group must always contain at least one compound element.
We  now show that, surprisingly, in a critical group every element different from $1$ must be compound.
\begin{theorem}\label{crit-impl-comp}
	If $G$ is a critical group, then $G$ is compound.
\end{theorem}
\begin{proof} Let $G$ be a critical group. 
	We prove the following claim.
	\begin{claim}
		Let $y\in G$ be plain and $p$ a prime such that $p\mid o(y)$. Then $p^2\mid o(y)$.
	\end{claim}
	Let $z\in \langle y \rangle$ be an element of order $p$. Since $G$ is critical, by Proposition \ref{ord-pl-cr}, $z$
	is compound critical and $[z]_{\mathtt{N}}\neq \mathcal{S}=\{1\}.$
	Thus, by Proposition \ref{par-crit-comp},
	$[z]_{\mathtt{N}}$ has parameters $(p,r,0)$, for a certain integer $r\geq2$.
	Let $x$ be a root of $[z]_{\mathtt{N}}$.
	Then we have $y\in N[z]=N[x]$, and hence $o(x)\mid o(y)$ or $o(y)\mid o(x)$.
	But $o(y) \mid o(x)$ cannot hold because $o(x)=p^r$ and, by Proposition \ref{ord-pl-cr},  $o(y)$ is not a prime power.
	Thus we have $p^r\mid o(y)$, and hence $p^2\mid o(y)$.
	\medskip
	
	Now suppose, by contradiction, that there exists a critical group  $G$ that is not compound.
	Then there exists at least a plain $y\in G\setminus \{1\}$.
	By Proposition \ref{ord-pl-cr} and  the claim above, we have that $o(y)=p^2q^2m$ for some distinct primes $p,q$  and  some $m\in \mathbb{N}$.	
	Let $z\in \langle y\rangle$ be an element of order $pq$. Since $o(z)$ is not a prime power, by Proposition \ref{ord-pl-cr}, we deduce that $z$ is  plain.  Thus, by the claim above, we get $p^2\mid  o(z)$, a contradiction.
\end{proof}
The above theorem has interesting consequences for the structure of critical groups. In particular, it has consequences for the Hughes-Thompson subgroup of the Sylow subgroups.
\begin{corollary}
	\label{TagliaGruppiCritici}
	Let $G$ be a critical group and $p$ a prime dividing $|G|.$  Then every element of order $p$ is the power of an element of order $p^2.$ In particular, $p^2\mid |G|$.  
\end{corollary}
\begin{proof}
	Let $x\in G$ with $o(x)=p$. By Theorem \ref{crit-impl-comp}, $x$ is compound and hence there exists $y\in G$ root for $[x]_{\mathtt{N}}$. Then  $o(y)=p^r$, $r\geq 2$, and $x$ is a power of $y.$ Since $p^2\mid o(y)$, we also have that $x$ is the power of an element of order $p^2.$
	\end{proof}

\begin{prop}\label{Hughes} Let $G$ be a critical group, $p$ a prime dividing $|G|$ and $P\in \mathrm{Syl}_p(G)$.  Then $H_p(P)=P.$
\end{prop}
\begin{proof} Assume, by contradiction, that $H_p(P)\neq P.$ Pick $x\in P\setminus H_p(P).$ Then $o(x)=p$. Since $G$ is critical, by Corollary \ref{TagliaGruppiCritici}, $x$ is the power of a certain $y\in G$ with $o(y)=p^2.$ It follows that $y\in H_p(P)$ and hence also $x\in H_p(P)$, a contradiction. 
\end{proof}
We are finally ready for a complete description of critical groups. We are going to see that critical groups are a subclass of the so-called EPPO groups, that is, those groups
whose  elements have prime power order.  The class of EPPO groups has a well consolidated history, starting with  Highman \cite{hig} who described the solvable case. Later Suzuki \cite{suz} described the simple case and the classification was finally completed by Brandl \cite{brandl}.

\begin{theorem}\label{Cr-iff-Frob} A group $G$ is 
critical if and only if  there exist distinct primes $p,q$ and integers $a,b\geq 2$  such that $G$ is a Frobenius group with kernel $F\cong C_{p^a}$ and complement $H\cong C_{q^b}$.
\end{theorem}
\begin{proof}  Let $G\neq 1$ be a critical group. Then $\mathcal{S}=\{1\}$ and $G$ cannot be cyclic.
Moreover, by Theorem \ref{crit-impl-comp}, $G$  is compound critical.
By Proposition \ref{ord-pl-cr}, every element of $G$ has prime power order  and thus $G$ is an EPPO group.
By Proposition \ref{par-crit-comp}, the $\mathtt{N}$-classes $C$ in $G$ and their neighbourhood closures $\hat{C}$ have the following form: $C=\langle y\rangle\setminus\{1\}$ and $\hat{C}=C \cupdot \{1\}=\langle y\rangle$ for some $y\in G$ with $o(y)=p^r$, where $p$ is a prime and $r\geq 2$ is an integer. 

Let $k$ be the number of non-trivial $\mathtt{N}$-classes. Observe that $k\geq 1$ because $G\neq 1.$
Let $y_1,\dots, y_k\in G$ be  roots for those classes. Such roots are, in particular, a system of representatives for the non-trivial $\mathtt{N}$-classes.
Thus, the distinct non-trivial $\mathtt{N}$-classes and their neighbourhood closures
are respectively given by $C_i:=[y_i]_{\mathtt{N}}=\langle y_i\rangle\setminus\{1\}$ and $\hat{C_i}=\langle y_i\rangle,$ for $i\in [k]$.

 Let $p_i^{\alpha_i}:=o(y_i)$ with $p_i$ prime and $\alpha_i\geq 2$, $i\in [k]$. 
Since $G=\left(\cupdot_{i=1}^k C_i\right) \cupdot\{1\},$ we have that $\{p_i:i\in[k]\}=\pi(G)$ and the set $\mathcal{P}:=\{\langle y_i\rangle : i\in [k]\}$ is a partition of $G$ by $k$ maximal cyclic subgroups of prime power order. 
Using that fact, we can rule out the possibility that $G$ is a $p$-group. Indeed, assume that $G$ is a $p$-group and observe that, by Proposition \ref{Hughes}, we have $H_p(G)=G.$ Thus, by Theorem \ref{Kegel}, we deduce that $\mathcal{P}$ is trivial, that is  $\mathcal{P}$ consists of a unique component equal to $G$. Hence $G$ is cyclic, a contradiction. 

We claim that every Sylow subgroup of $G$ is cyclic and appears in $\mathcal{P}$.
 Indeed, let $p\in \pi(G)$ and $P\in \mathrm{Syl}_p(G)$. Then, we have $$P=\bigcup_{i=1}^k(P\cap \langle y_i\rangle),$$ and at least one of the subgroups of $P$ given by $P\cap \langle y_i\rangle$ must be non-trivial. Reordering if necessary, we assume that $P\cap \langle y_i\rangle\neq 1$ if and only if $i\in [\ell]$ for a suitable  $1\leq \ell\leq k$. 
Since for $i\neq j\in [\ell]$, we have $(P\cap \langle y_i\rangle)\cap (P\cap \langle y_j\rangle)\leq \langle y_i\rangle \cap \langle y_j\rangle=1,$ we deduce that 
$$P=\bigcup_{i=1}^{\ell}(P\cap \langle y_i\rangle)$$ realizes a partition of $P$ in $\ell$ components. But, by Proposition \ref{Hughes}, we have $H_p(P)=P$. Thus, by Theorem \ref{Kegel}, the only partition of $P$ is the trivial partition. This means that $\ell=1$ and $P\cap \langle y_1\rangle=P$.  Thus $\langle y_1\rangle\geq P$. As a consequence $P$ is cyclic.
Moreover, we necessarily have $P=\langle y_1\rangle.$ Indeed $o(y_1)$ is a prime power and hence, it is a $p$-power. Thus $\langle y_1\rangle$ cannot properly contain the $p$-Sylow subgroup $P.$

We claim now that the components of $\mathcal{P}$ having order a power of a same $p \in\pi(G)$ are exactly the collection of the $p$-Sylow subgroups of $G.$
Let $\langle y_i\rangle$  be such that $o(y_i)=p_i^{\alpha_i}$, with $p_i=p$. Then $y_i\in P$ for some $P\in  \mathrm{Syl}_p(G)$. Since we have observed that every Sylow of $G$ appears in $\mathcal{P}$, we deduce that there exists $j\in [k]$ such that $P=\langle y_j\rangle\geq \langle y_i\rangle$ which, by the definition of partition,
implies $\langle y_j\rangle=\langle y_i\rangle$ and thus
 $i=j$. It follows that $\langle y_i\rangle=P.$ 

As a consequence, for a fixed $p \in\pi(G)$, the roots $ y_i$ having order a power of a certain prime $p$ have all the same order $p^{n(p)}$, $n(p)\geq 2$, and
 the partition $\mathcal{P}$ is the collection of the distinct Sylow subgroups of $G.$

Now,  a classic theorem by H\"older, Burnside, Zassenhaus  (\cite[10.1.10]{Rob}) characterizes the  finite groups with cyclic Sylow subgroups.  From that theorem and recalling that $G$ is a non-cyclic EPPO group with $P\cong C_{p^{n(p)}}$, $n(p)\geq 2$, for all $p\in \pi(G)$, we deduce that $G$ is a meta-cyclic non-cyclic  group with presentation
\begin{equation}\label{forma}
G=\langle x, y: x^{p^a}=1,y^{q^b}=1, x^y=x^r \rangle
\end{equation}
where  $p$ and $q$ are distinct primes, $a,b$ are integers with  $a, b\geq 2$, $2\leq r<p^a$, $p\nmid r-1$  and  $r^{q^b}\equiv 1(\mathrm{mod}\, p^a)$. Note that this last condition implies $p\nmid r.$ In particular, we have $ \pi(G)=\{p,q\}$ and $|G|=p^aq^b$. Then $\mathcal{P}=\{P, Q_1,\dots, Q_{n_q}\}$ where $P=\langle x\rangle$ is the unique $p$-Sylow of $G$ and $Q_1,\dots, Q_{n_q}$ are the distinct $q$-Sylow subgroups of $G$. Then we have 
$$p^aq^b=p^a+n_q(q^b-1)$$
which implies $p^a=n_q=|G:N_G(Q)|$ for every $q$-Sylow subgroup $Q$ of $G$. As a consequence, we have  $Q=N_G(Q)$. Let now $g\in G\setminus Q$. Then $Q^g\neq Q$ and thus, since $\mathcal{P}$ is a partition,  $Q^g\cap Q=1$. Hence $G$ is a Frobenius group with complement $Q\cong C_{q^b}$ and kernel $P\cong C_{p^a}$.
\medskip

Conversely, let $G$ be a Frobenius groups with Frobenius kernel $F\cong C_{p^a}$ and Frobenius complement $H\cong  C_{q^b},$ where $p$ and $q$ are distinct primes and $a,b\geq 2$ are integers. We show that $G$ is compound critical. As well known (\cite[8.5.5]{Rob}),
$G$ admits the partition $$\mathcal{P}:=\{F, H^x: x\in F\}$$
into $|\mathcal{P}|=p^a+1$ components given by the kernel and the distinct complements. By our assumption, all the components in $\mathcal{P}$ are cyclic of prime power order given by $p^a$ or $q^b$. Then, by Proposition \ref{applicazione_critical}, $G$ is critical.
 \end{proof}   



 


Note that there exists no critical group with kernel given by a $p$-group with $p\in \{2,3\}$. Indeed, by Proposition \ref{Frobenius-espliciti}\, $(ii)$, we know that the existence of a Frobenius group with kernel $F\cong C_{p^a}$ and complement $H\cong C_{q^b}$, where $a,b\geq 2$, requires $q^b\mid p-1.$
Now, clearly, there exist no prime $q$ and no $b\geq 2$ such that $q^b\mid 2-1=1$ or such that $q^b\mid 3-1=2.$ 
In particular we have that $p$ must be odd.
Additionally, when $p=5$ the only choice for $q^b$ is given by $q=2=b$. Choosing $a=2$ and $r=7$, we 
construct the critical group $C_{25}\rtimes C_{4}$ of minimum size $100,$ defined by the condition $x^y=x^7$.  Note that this is the group of minimum size in the family described in Example \ref{Esempio_gruppo_critico}.
Moreover, observe that all the critical groups are solvable but non-nilpotent.
\smallskip
Collecting Theorem \ref{Cr-iff-Frob} and Proposition \ref{Frobenius-espliciti}, we obtain our Main Theorem.
\begin{main}
A group is critical if and only if $G=\langle x\rangle\rtimes \langle y\rangle$, where $\langle x\rangle=C_{p^a}$, $\langle y\rangle=C_{q^b}$, with $p\neq q$ primes, $a,b\geq 2$, $x^y=x^r,$ with $2\leq r<p^a$ such that $p\nmid r$ and $q^b=|r|_{\mod p}$.
\end{main}

\section{Critical groups, partitions and EPPO groups}

In this section we show some applications of the results obtained in the previous sections.
First of all, Theorem \ref{Cr-iff-Frob} allows us to enhance the content of Proposition
\ref{applicazione_critical} in the case of cyclic components of order a prime power with exponent at least $2$,
 shedding new light on partitions.
\begin{maincor} Let $G$ be a group 
	admitting a non-trivial partition into cyclic subgroups of orders  a prime power with exponent at least $2$. Then $\pi(G)=\{p,q\}$ with $p,q$ distinct primes and $G$ is a Frobenius group with kernel $F\cong C_{p^a}$ and complement $H\cong C_{q^b}$, for suitable $a,b\geq2$ and $p$ odd.
\end{maincor} 
The Main Theorem makes clear that the critical groups are particular non-cyclic metacyclic EPPO groups with $|\pi(G)|\geq 2$. Those groups are characterized in \cite{Shi}.\footnote{ \cite[Theorem 1.3]{Shi} is stated in a different way and some tedious typos occur throughout the paper. However, it is clear that what it is proven there is the content of Theorem \ref{shi}.} 

\begin{theorem}{\rm \cite[Theorem 1.3]{Shi}}\label{shi} Let $G$ be a group with $|\pi(G)|\geq 2$.
The following facts are equivalent.
\begin{itemize}
\item[$(i)$] $G$ is a non-cyclic metacyclic  {\rm EPPO}-group;
\item[$(ii)$]$G=\langle x,y\  |\  x^{p^a}=1,\  y^{q^b}=1,\  x^y =x^r\rangle,$
for some $p$ and $q$ distinct primes, $a,b,r\in \mathbb{N}$ with $p\nmid r$ and $q^b=|r|_{\mod p^a}.$ 
\end{itemize}
\end{theorem}

We call the parameters $p, q, a, b, r$ describing the group $G$ in the above theorem the parameters of the EPPO non-cyclic metacyclic group $G$.

By Proposition \ref{Frobenius-espliciti} and Theorem \ref{shi}, we now deduce a characterization of the EPPO non-cyclic metacyclic Frobenius groups.
\begin{prop}\label{shi-Frob} Let $G$ be a group with $|\pi(G)|\geq 2$. The following facts are equivalent.
\begin{itemize}
\item[$(i)$] $G$ is a non-cyclic metacyclic Frobenius  {\rm EPPO} group;
\item[$(ii)$]$G=\langle x,y\  |\  x^{p^a}=1,\  y^{q^b}=1,\  x^y =x^r\rangle,$
where $p$ and $q$ are different primes, $a,b,r\in \mathbb{N}$ and $q^b=|r|_{\mod p}.$
\end{itemize}
\end{prop}

\begin{proof}
	Assume first that $(i)$ holds. 
	 Then, by Theorem \ref{shi}, $G=\langle x,y\  |\  x^{p^a}=1,\  y^{q^b}=1,\  x^y =x^r\rangle,$ for some $p$ and $q$ distinct primes, $a,b,r\in \mathbb{N}$ with $p\nmid r$ and $q^b=|r|_{\mod p^a}.$ Thus $G$ is a semidirect product of cyclic groups as those examined in Proposition  \ref{Frobenius-espliciti}. 
	Since in a Frobenius group the kernel $K$ is given by the Fitting set and $\langle x\rangle$ is nilpotent and normal, we have that $K\supseteq \langle x\rangle$. Moreover, the Fitting subgroup cannot be larger because otherwise we would have $q\mid|K|$ and thus $\gcd(|K|, |G/K|)\neq 1$, against the coprimality of kernel and complement in a Frobenius group. Thus the kernel is given by $\langle x\rangle$ and $G$ is a Frobenius group of the type described by Proposition \ref{Frobenius-espliciti}\,$(i)$. As a consequence, we have $q^b=|r|_{\mod p}$.
	\medskip
	
	Conversely assume that $(ii)$ holds. Then, by Proposition \ref{Frobenius-espliciti}\, $(i)$, $G$ is a Frobenius group with kernel $\langle x\rangle\cong C_{p^a}.$ Clearly it is also metacyclic and non-cyclic. Since the complement is isomorphic to $C_{q^b},$ taking into account the partition of $G$ we have that $G$ is also an EPPO group.
\end{proof}
We conclude with a special characterization of the EPPO non-cyclic metacyclic Frobenius groups  having parameters $a,b \geq 2$.
\begin{prop}\label{crit-eppo} Let $G$ be a non-cyclic metacyclic  {\rm EPPO} group with $|\pi(G)|\geq 2$ and parameters $a, b\geq 2.$ Then the following facts are equivalent:
\begin{itemize}
\item[$(i)$] $G$ is a Frobenius group;
\item[$(ii)$] $G$ is critical.
\end{itemize}
\end{prop}
\begin{proof} $(i)\Rightarrow (ii)$ Let $G$ be a Frobenius group. Since $G$ is a non-cyclic metacyclic  {\rm EPPO} group with $|\pi(G)|\geq 2$,  by Proposition \ref{shi-Frob}, we have that $G=\langle x,y\  |\  x^{p^a}=1,\  y^{q^b}=1,\  x^y =x^r\rangle,$
where $p$ and $q$ are different primes and $q^b=|r|_{\mod p}.$ Thus, as explained inside the proof of Proposition \ref{shi-Frob}, the kernel of $G$ is given by $\langle x\rangle\cong C_{p^a}$ and a complement of $G$ by $\langle y\rangle\cong C_{q^b}$. Since $a, b\geq 2$, by  Theorem \ref{Cr-iff-Frob}, we deduce that $G$ is critical.

\medskip
$(ii)\Rightarrow (i)$ Assume that $G$ is critical. Then by Theorem \ref{Cr-iff-Frob}, $G$ is a Frobenius group.
\end{proof}

\vspace{7mm}

\noindent {{\bf Acknowledgments}}   Daniela Bubboloni has been supported by GNSAGA of INdAM (Italy),  by the European Union - Next Generation EU, Missione 4 Componente 1, PRIN 2022-2022PSTWLB - Group Theory and Applications, CUP B53D23009410006, and by local funding from the Universit\`a degli Studi di Firenze. Both authors are indebted to the participants of the International Joint Meeting of the Unione Matematica Italiana (UMI) and the American Mathematical Society (AMS), Special Session Graphs Associated with Groups: Advances and Applications, held in Palermo in July 2024, for many useful comments on a preliminary version of the paper.

		\vspace{10mm}
\noindent {\Large{\bf Conflict of interest}}
\vspace{2mm}

\noindent Declarations of conflict of interest in the manuscript: none.

	\end{document}